\documentclass{amsart}

\usepackage{amsfonts,amssymb,amsmath}
\usepackage{graphicx}
\usepackage{psfrag}
\usepackage{color}
\usepackage{verbatim}
\usepackage{epstopdf}

\newtheorem{theo}{Theorem}[section]
\newtheorem{lem}[theo]{Lemma}
\newtheorem{prop}[theo]{Proposition}
\newtheorem{cor}[theo]{Corollary}

\theoremstyle{definition}

\theoremstyle{remark}
\newtheorem{rem}[theo]{Remark}
\newtheorem{ques}{Question}[section]
\newtheorem{exam}{Example}[section]
\newtheorem{conj}{Conjecture}[section]

\numberwithin{equation}{section}

\newcommand{\bR}{{\mathbb R}}
\newcommand{\bZ}{{\mathbb Z}}

\def\gep{\varepsilon}

\def\gl{\lambda}

\def\wdt{\widetilde}
\def\wdtM{\widetilde{M}}

\def\upBdim{\overline{\dim}_B}
\def\lowBdim{\underline{\dim}_B}
\def\Bdim{\dim_B}

\def\N{\mathcal{N}}

\begin{document}

\title[Box dimension of generalized affine FIFs]{Box dimension of generalized affine \\ fractal interpolation functions}

\author{Lai Jiang}
\address{School of Mathematical Sciences, Zhejiang University, Hangzhou 310027, China}
\email{jianglai@zju.edu.cn}

\author{Huo-Jun Ruan}
\address{School of Mathematical Sciences, Zhejiang University, Hangzhou 310027, China}
\email{ruanhj@zju.edu.cn}
\thanks{The research was supported in part by NSFC grant 11771391, ZJNSF grant LY22A010023 and the Fundamental Research Funds
for the Central Universities of China grant 2021FZZX001-01.}

\thanks{Corresponding author: Huo-Jun Ruan}


\date{}


\begin{abstract}

  Let $f$ be a generalized affine fractal interpolation function with vertical scaling function $S$. In this paper, we study $\Bdim \Gamma f$, the box dimension of the graph of $f$, under the assumption that $S$ is a Lipschtz function. By introducing vertical scaling matrices, we estimate the upper bound and the lower bound of oscillations of $f$. As a result, we obtain explicit formula of $\Bdim \Gamma f$ under certain constraint conditions.
\end{abstract}

\maketitle

\section{Introduction}
\label{intro}

Let $N\geq 2$ be a positive integer. Given a data set $\{(x_i,y_i)\}_{i=0}^N\subset \bR^2$ with $x_0<x_1<\ldots<x_N$, there are many classical methods to construct functions interpolating the data set, while interpolation functions are smooth or piecewise smooth. In 1986, Barnsley \cite{Bar86} introduced fractal functions to interpolate the data set.

Let $L_i:\, [x_0,x_N]\to [x_{i-1},x_i], 1\leq i\leq N$ be contractive homeomorphisms with
\begin{equation}\label{eq:1-1}
  L_i(x_0)=x_{i-1}, \quad L_i(x_N)=x_i.
\end{equation}
Let $F_i:\, [x_0,x_N]\times \bR \to \bR, 1\leq i\leq N$ be continuous maps satisfying
\begin{equation}\label{eq:1-2}
  F_i(x_0,y_0)=y_{i-1}, \quad F_i(x_N,y_N)=y_i,
\end{equation}
and $F_i$ is uniformly contractive with the second variable, i.e., there exists a constant $\beta_i\in (0,1)$, such that for all $x\in [x_0,x_N]$, and all $y',y''\in \bR$,
\begin{equation}\label{eq:1-3}
  |F_i(x,y')-F_i(x,y'')| \leq \beta_i|y'-y''|.
\end{equation}
Then we can define maps $W_i:\, [x_0,x_N]\times \bR \to [x_{i-1},x_i]\times \bR$, $1\leq i\leq N$ by
\begin{equation}\label{eq:1-4}
  W_i(x,y)=(L_i(x),F_i(x,y)).
\end{equation}
From above conditions, it is easy to check that $W_i(x_0,y_0)=(x_{i-1},y_{i-1})$ and $W_i(x_N,y_N)=(x_i,y_i)$ for each $i$.

Notice that for each $1\leq i\leq N$, $W_i$ is continuous and it maps $[x_0,x_N]\times \bR$ into itself. Hence $\{W_i:\, 1\leq i\leq N\}$ is an \emph{iterated function system} (IFS for short) on $[x_0,x_N]\times \bR$.
Barnsley \cite{Bar86} proved that
there exists a unique continuous function $f$ on $[x_0,x_N]$ such that its graph $\Gamma f:=\{(x,f(x)):\, x\in [x_0,x_N]\}$ is the invariant set of the IFS $\{W_i:\, 1\leq i\leq  N\}$, i.e.,
\begin{equation}\label{eq:1-5}
  \Gamma f=\bigcup_{i=1}^N W_i(\Gamma f).
\end{equation}
Furthermore, the function $f$ always interpolates the data set, i.e., $f(x_i)=y_i$ for all $1\leq i\leq N$.
The function $f$ is called the \emph{fractal interpolation function} (FIF for short) determined by the IFS $\{W_i\}_{i=1}^N$.

Notice that box dimension is one of the most important dimensions in fractal geometry and its applications. Thus it is quite natural to study $\Bdim \Gamma f$, where $f$ is an FIF.

%


In the case that every $W_i$ is an affine maps, we call $f$ an \emph{affine FIF}. In this case, for each $i$, there exist real numbers $a_i,b_i,c_i,d_i$ and $e_i$, such that
\[
  W_i(x,y)=(a_i x+ b_i, c_i x+d_i y+e_i).
\]
$d_i$'s are called \emph{vertical scaling factors} of $f$. According to \eqref{eq:1-3}, $|d_i|<1$ for each $i$.
In \cite{BEHM89}, Barnsley, Elton, Hardin and Massopust obtained the box dimension formula of affine FIFs. They proved that if $\sum_{i=1}^N |d_i|>1$ and the interpolation points $\{(x_i,y_i)\}_{i=0}^N$ are not collinear, then the box dimension of $\Gamma f$ equals the unique real number $s$ satisfying the following equation:
\[
  \sum_{i=1}^N a_i^{s-1} |d_i|=1;
\]
and $\Bdim \Gamma f=1$ otherwise. This formula can be generalized to recurrent affine FIFs, see \cite{BEH89,RXY21} for example.


It is easy to check that if $f$ is an affine FIF, then $F_i$ can be rewritten as
\[
  F_i(x,y)=d_i (y-b(x)) + h(L_i(x)),
\]
where
\begin{itemize}
\item[(A1)] $b$ is a linear function satisfying $b(x_0)=y_0$ and $b(x_N)=y_N$,
\item[(A2)] $h$ is a piecewise linear function satisfying $h(x_i)=y_i$, $0\leq i\leq N$, and $h|_{[x_{i-1},x_i]}$ is linear for each $1\leq i\leq N$,
\item[(A3)] $L_i$, $1\leq i\leq N$, are linear functions on $[x_0,x_N]$ satisfying \eqref{eq:1-1}.
\end{itemize}

Now let $S(x)$ be a continuous function on $[x_0,x_N]$ with $|S(x)|<1$ for all $x\in [x_0,x_N]$. For each $1 \leq i\leq N$, we define
\begin{equation}\label{eq:1-6}
  F_i(x,y)=S(L_i(x)) (y-b(x)) +h(L_i(x)),\quad i=1,2,\ldots,N,
\end{equation}
where conditions (A1)-(A3) are satisfied. Then it is easy to see that $F_i$ satisfies \eqref{eq:1-2} and \eqref{eq:1-3}. Thus, if we define $W_i$ by \eqref{eq:1-4}, then $\{W_i\}_{i=1}^N$ determines an FIF $f$. In this case, we call $f$ a \emph{generalized affine FIF}, and call $S$ the \emph{vertical scaling function} of $f$.



In general, it is very challenging to obtain the box dimension of generalized affine FIFs without any restrictions.
Till now, as we know, there are few results in this direction. In \cite{BaMa15}, Barnsley and Massopust studied a special case of generalized affine FIFs. They assumed that the vertical scaling function $S$ is linear on $[x_{i-1},x_i]$ for all $1\leq i\leq N$. In this case, the corresponding FIF $f$ is called a \emph{bilinear FIF}. With some additional conditions, they obtained the box dimension formula of bilinear FIFs in the case of equally spaced data points. We remark that essentially, the proof in \cite{BaMa15} need the following condition: $S$ is nonnegative and it has uniform sum, that is, $\gamma(x)=\sum_{i=1}^N S(L_i(x))$ is constant on $[x_0,x_N]$. See 
Remark~\ref{rem:bilinear} and Remark~\ref{rem:4-13} for more details. This work was generalized to bilinear fractal interpolation surfaces on rectangular grids \cite{KRZ18}.

In this paper, we study the box dimension of generalized affine FIFs without the assumption that $S$ has the uniform sum. More precisely, we require that the following conditions are satisfied:
\begin{itemize}
  \item[(A4)] $\{x_i\}_{i=0}^N$ are uniformly spaced on $[x_0,x_N]$, that is, $x_i-x_{i-1}=(x_N-x_0)/N$ for all $1\leq i\leq N$,
  \item[(A5)] $S$ is a Lipschitz function, that is, there is a constant $\lambda_S>0$, such that $|S(x')-S(x'')|\leq \lambda_S |x'-x''|$ for all $x',x''\in [x_0,x_N]$,
  \item[(A6)] $S$ is positive on $I$, that is, $S(x)>0$ for all $x\in [x_0,x_N]$.
\end{itemize}
We prove that if the conditions (A1)-(A6) are satisfied, then $\Bdim \Gamma f$ equals either $1$ or $1+\log (\rho_S)/\log N$, where $\rho_S$ is a constant dependent on the function $S$. See Section~3 for the explicit definition of $\rho_S$.

We remark that there are many applications of FIFs, see  \cite{MaHa92,VDL94} for examples. Also, there are many works on the box dimension and Hausdorff dimension of fractal interpolation functions and fractal interpolation surfaces. We refer the readers to
\cite{BRS20,BDD06,Fen08,LR21} and the references therein. Other properties of FIFs have been studied in various papers. Please see \cite{CKV15,Nav10,WY13} for examples.
 

The paper is organized as follows. In section~2, we recall the definition of box dimension and present some properties of generalized affine FIFs. In Section~3, we introduce two sequences of vertical scaling matrices, and prove that the limits of spectral radii of these two sequences of matrices coincide under certain conditions. By using this result, in Section~4, we estimate the upper bound and the lower bound of oscillations of generalized affine FIFs, and obtain explicit formula of their box dimension under certain conditions. In Section~5, we present an example to explain our result. We also make some further remarks in this section.

\section{Preliminaries}\label{sec:FIF}

\subsection{Definition of box dimension}
For any $k_1, k_2\in\mathbb{Z}$ and $\varepsilon>0$, we call $[k_1\varepsilon,(k_1+1)\varepsilon] \times [k_2\varepsilon,(k_2+1)\varepsilon]$ an $\varepsilon$-coordinate square in $\mathbb{R}^2$. Let $E$ be a bounded set in $\mathbb{R}^2$ and $\mathcal{N}_E(\varepsilon)$ the number of $\varepsilon$-coordinate squares intersecting $E$. We define
\begin{equation}\label{eq:box-dim-def}
	\overline{\dim}_{B} E=\varlimsup_{\gep\to 0+}\frac{\log \mathcal{N}_{E}(\gep)}{\log1/\gep}
	\quad\text{and}	\quad
	\underline{\dim}_{B} E=\varliminf_{\gep\to 0+}\frac{\log \mathcal{N}_{E}(\gep)}{\log1/\gep},
\end{equation}
and call them the \emph{upper box dimension} and the \emph{lower box dimension} of $E$. Respectively, if $\overline{\dim}_{B} E=\underline{\dim}_{B} E$, then we use  $\dim_B E$ to denote the common value and call it the \emph{box dimension} of $E$.
It is easy to see  that in the definition of the upper and lower box dimensions, we can only consider $\gep_k=\frac{1}{N^k}$, where $k \in \mathbb{Z}^+$. That is,
\begin{equation}\label{eq:box-dim-def2}
  	\overline{\dim}_{B} E=\varlimsup_{k \to \infty}\frac{\log \mathcal{N}_{E}(\gep_{k})}{k\log N}
  	\quad\text{and}\quad
  	\underline{\dim}_{B} E=\varliminf_{k\to \infty}\frac{\log \mathcal{N}_{E}(\gep_{k})}{k\log N}.
\end{equation}

It is also well known that $\underline{\dim}_{B}E\geq 1$ when $E$ is the graph of a continuous function on a closed interval of $\mathbb{R}$. Please see ~\cite{Fal90} for details.

\subsection{Some properties of generalized affine FIFs}

In this subsection, we assume that $f$ is a generalized affine FIF satisfying conditions (A1)-(A3) and \eqref{eq:1-6}.  From \eqref{eq:1-5}, for each $1\leq i\leq N$,
\[
  W_i(x,f(x))=(L_i(x),f(L_i(x)))
\]
so that $f(L_i(x))=F_i(x,f(x))$. Combining this with \eqref{eq:1-6},
\begin{equation}\label{eq:2-1}
  f(L_i(x))= S(L_i(x)) (f(x)-b(x)) +h(L_i(x))
\end{equation}
for all $x\in [x_0,x_N]$ and $1\leq i\leq N$.

Let $g$ be the linear function on $[0,1]$ satisfying $g(0)=x_0$ and $g(1)=x_N$. Then $g^{-1}$ is a linear function on $[x_0,x_N]$. Write $t_i=g^{-1}(x_i)$ for $0\leq i\leq N$. Define
\begin{equation}\label{eq:def-f*}
 f^*(t)=f(g(t))-b(g(t)), \quad t\in [0,1].
\end{equation}

The following result is well-known. However, we have not seen the proof in other papers. Hence, we present the proof for readers' convenience.
\begin{lem}
 For each $1\leq i\leq N$, let $L^*_i$ be the linear function on $[0,1]$ satisfying $L^*_i(0)=t_{i-1}$ and $L^*_i(1)=t_{i}$. Furthermore, we define
 \begin{equation}\label{eq:lem2-1}
   F^*_i(t,y)=S^*(L^*_i(t)) y +h^*(L^*_i(t)), \quad t\in [0,1], y\in \bR,
 \end{equation}
 where $S^*(t)=S(g(t))$ and $h^*(t)=h(g(t))-b(g(t))$ for $t\in [0,1]$.
 Then $f^*$ is a generalized affine FIF determined by the IFS $\{(L^*_i(t), F^*_i(t,y))\}_{i=1}^N$.
\end{lem}
\begin{proof}
  Fix $1\leq i\leq N$. Notice that both $g\circ L_i^*$ and $L_i\circ g$ are linear functions on $[0,1]$. Thus, from
  \begin{align*}
    g(L_i^*(0))=x_{i-1}= L_i( g(0)),\quad
    g(L_i^*(1))=x_{i}= L_i( g(1)),
  \end{align*}
  we have $g\circ L_i^*=L_i\circ g$ on $[0,1]$.
  Now, given $t\in [0,1]$, we write $x=g(t)$. Then $g( L_i^*(t))=L_i(x)$. By definitions of $S^*$ and $h^*$,
  \begin{align*}
  &\; S^*(L^*_i(t)) f^*(t) +h^*(L^*_i(t)) \\
    =&\; S\big(g(L_i^*(t))\big)\Big( f(g(t))-b(g(t)) \Big) + \Big( h\big(g(L_i^*(t))\big) - b\big(g(L_i^*(t))\big)\Big)\\
    =&\; S(L_i(x))\big( f(x)-b(x) \big) + h(L_i(x)) - b(L_i(x)) \\
    =&\;  f(L_i(x) ) -b(L_i(x)),
  \end{align*}
  where the last equality follows from \eqref{eq:2-1}.
  Thus, from \eqref{eq:def-f*} and \eqref{eq:lem2-1},
  \begin{align*}
    f^*(L_i^*(t))&=f\big(g( L_i^*(t))\big) - b\big(g( L_i^*(t))\big)=f(L_i(x))-b(L_i(x)) \\
    &=S^*(L^*_i(t)) f^*(t) +h^*(L^*_i(t))=F_i^*(t,f^*(t)),
  \end{align*}
  which implies that $\Gamma f^*=\{(t,f^*(t)):\, t\in [0,1]\}$ is the invariant set of $\{W_i^*\}_{i=1}^N$, where $W_i^*(t,y)=(L_i^*(t),F_i^*(t,y))$. This completes the proof of the lemma.
\end{proof}


Notice that both $g$ and $b$ are linear function. Thus, it is easy to check that
\[
  \varphi\big((t,f^*(t))\big)=\big(g(t),f(g(t))\big),\quad t\in [0,1],
\]
is a bi-Lipschitz map from $\Gamma f^*$ to $\Gamma f$. Hence
\[
  \lowBdim \Gamma f^*=\lowBdim \Gamma f, \quad \upBdim \Gamma f^*=\upBdim \Gamma f.
\]

Since $f^*(0)=f^*(1)=0$, in the sequel of this paper, we always assume that $x_0=0$, $x_N=1$ and $y_0=y_N=0$. From now on, we write $I=[0,1]$, and $I_i=[\frac{i-1}{N},\frac{i}{N}]$ for all $1\leq i\leq N$. In this case, $b(x)\equiv 0$ on $I$. Hence, from \eqref{eq:2-1}, we have the following useful property:
\begin{equation}\label{eq:FIF-rec}
  f(x)=S(x) f(L_i^{-1}(x))+h(x)
\end{equation}
for all $x\in I_i$, where $1\leq i\leq N$.

\section{Vertical scaling matrices}

In this section, we assume that the vertical scaling function $S$ satisfies the condition (A5), i.e., there exists a constant $\gl_S>0$, such that
\begin{equation}\label{eq:3-1}
	|S(x')-S(x'')|\leq \lambda_S|x'-x''|,	\quad x',x'' \in I.
\end{equation}


Given a closed interval $E=[a,b]$, for each $k\in \mathbb{Z}^+$ and $1\leq j\leq N^k$, we write
\begin{equation*}
	E_j^k=\Big[a+\frac{j-1}{N^k}(b-a),a+\frac{j}{N^k}(b-a)\Big].
\end{equation*}
It is clear that $I_j^k=[\frac{j-1}{N^k},\frac{j}{N^k}]$. For simplicity, for each $1\leq i \leq N$, we write
\begin{equation*}
	I_{i,j}^k=(I_i)_j^k=\Big[\frac{i-1}{N}+\frac{j-1}{N^{k+1}},\frac{i-1}{N}+\frac{j}{N^{k+1}}\Big].
\end{equation*}




\subsection{Two sequences of vertical scaling matrices}

Let $k$ be a positive integer. Given $1\leq i \leq N$ and $1\leq j \leq N^k$, we define
\[
  \overline{s}_{i,j}^k=\max_{x \in I_{i,j}^k} |S(x)|.
\]
It is clear that $\overline{s}_{i,j}^k=\max_{x \in I_{j}^k} |S(L_i(x))|.$
In order to calculate the box dimension of FIF, we introduce an $N^k \times N^k $ matrix $M_k$ as follows,
\begin{align*}
	\begin{pmatrix}
	\overline{s}_{1,1}^k & \cdots & \overline{s}_{1,N}^k & & & & & & & \\
	 & & &\overline{s}_{1,N+1}^k & \cdots & \overline{s}_{1,2N}^k & & & & \\
	 & & & & & & \ddots & & & \\
	& & & & & & &\overline{s}_{1,{N}^{k}-N+1}^k & \cdots & \overline{s}_{1,{N}^{k}}^k  \\
	\overline{s}_{2,1}^k & \cdots & \overline{s}_{2,N}^k & & & & & & & \\
	 & & & \overline{s}_{2,N+1}^k & \cdots & \overline{s}_{2,2N}^k & & & & \\
	 & & & & & & \ddots & & & \\
	& & & & & & &\overline{s}_{2,{N}^{k}-N+1}^k & \cdots & \overline{s}_{2,{N}^{k}}^k  \\
	\vdots & & \vdots & \vdots & & \vdots &  & \vdots & & \vdots \\
	\overline{s}_{N,1}^k & \cdots & \overline{s}_{N,N}^k & & & & & & & \\
	 & & &\overline{s}_{N,N+1}^k & \cdots & \overline{s}_{N,2N}^k & & & & \\
	 & & & & & & \ddots & & & \\
	& & & & & & &\overline{s}_{N,{N}^{k}-N+1}^k & \cdots & \overline{s}_{N,{N}^{k}}^k  \\
	\end{pmatrix}.
\end{align*}
That is, for $1\leq i\leq N$, $1\leq \ell\leq N^{k-1}$ and $1\leq j\leq N^k$,
\begin{equation}\label{eq:Mk-def}
	(M_k)_{(i-1)N^{k-1}+\ell,j}=\begin{cases}
    \overline{s}^k_{i,j}, & \mbox{if } (\ell-1)N< j\leq \ell N, \\
    0, & \mbox{otherwise}.
  \end{cases}
\end{equation}

Similarly, we define $\underline{s}_{i,j}^k=\min_{x \in I_{i,j}^k} |S(x)|$ and define another $N^k \times N^k$ matrix $M'_k$ by
replacing $\overline{s}_{i,j}^k$ with $\underline{s}_{i,j}^k$  in \eqref{eq:Mk-def}. Both $M_k$ and $M'_k$ are called \emph{vertical scaling matrices} with level-$k$.




Now we recall some notations and definitions in matrix analysis \cite{HorJoh90}.
Given a matrix $X=(X_{ij})_{n\times n}$, we say $X$ is \emph{nonnegative} (resp. \emph{positive}), denoted by $X\geq 0$ (resp. $X>0$), if $X_{ij}\geq0$ (resp. $X_{ij}>0$) for all $i$ and $j$. Let $Y=(Y_{ij})_{n\times n}$ be another matrix. We write $X\geq Y$ (resp. $X>Y$) if $X_{ij}\geq Y_{ij}$ (resp. $X_{ij}>Y_{ij}$) for all $i$ and $j$. Similarly, given $u=(u_1,\ldots,u_n),v=(v_1,\ldots,v_n)\in \bR^n$, we write $u\geq v$ (resp. $u>v$) if $u_i\geq v_i$ (resp. $u_i>v_i$) for all $i$.

A nonnegative matrix $X=(X_{ij})_{n\times n}$ is called \emph{irreducible} if for any $i,j\in \{1,\ldots,n\}$, there exists a finite sequence $i_0,\ldots,i_t\in \{1,\ldots,n\}$, such that $i_0=i,i_t=j$ and $X_{i_{\ell-1},i_\ell}>0$ for all $1\leq \ell \leq t$. $X$ is called \emph{primitive} if there exists $k\in \bZ^+$, such that $X^k>0$. It is clear that a primitive matrix is irreducible.

The following lemma is well known. Please see \cite[Chapter 8]{HorJoh90} for details.
\begin{lem}[Perron-Frobenius Theorem]\label{th:PF}
Let $X=(X_{ij})_{n\times n}$ be an irreducible nonnegative matrix. Then
\begin{enumerate}
	\item $\rho(X)$, the spectral radius of $X$, is positive,
	\item $\rho(X)$ is an eigenvalue of $X$ and has a  positive eigenvector,
	\item $\rho(X)$ increases if any element of $X$ increases.
\end{enumerate}
\end{lem}

\begin{lem}\label{lem:primitive}
  Assume that the vertical scaling function $S$ is not identically zero on every subinterval of $I$. Then
  $(M_k)^k>0$ for all $k\in \bZ^+$. As a result, $M_k$ is primitive for all $k\in \bZ^+$.
\end{lem}
\begin{proof}
  Let $k\in \bZ^+$. By the assumption of the lemma, it is clear that  $\overline{s}_{i,j}^k>0$ for all $1\leq i\leq N$ and $1\leq j\leq N^k$.
  Now, for any $j,\ell\in \{1,2,\ldots,N^k\}$, there exist $j_1,\ldots,j_k,\ell_1,\ldots,\ell_k \in \{1,\ldots,N\}$, such that
  \begin{align*}
    j&=(j_1-1)N^{k-1}+(j_2-1)N^{k-2}+\cdots+(j_{k-1}-1)N+j_k,\\
    \ell&=(\ell_1-1)N^{k-1}+(\ell_2-1)N^{k-2}+\cdots+(\ell_{k-1}-1)N+\ell_k.
  \end{align*}
  Define $t_1=j$ and
  \[
    t_{p+1}=N(t_p-(j_p-1)N^{k-1}-1)+\ell_p, \quad 1\leq p\leq k.
  \]
  Then it is easy to see that $t_{k+1}=\ell$. From the definition of the matrix $M_k$, it is easy to see that
  $
    (M_k)_{t_p,t_{p+1}}>0
  $
  for all $1\leq p\leq k$. Thus
  \[
    \big((M_k)^k\big)_{j,\ell}\geq \prod_{p=1}^k (M_k)_{t_p,t_{p+1}}>0.
  \]
  By the arbitrariness of $j$ and $\ell$, the lemma holds.
\end{proof}

\begin{theo}\label{th:rho-M}
Assume that the the vertical scaling function $S$ is not identically zero on every subinterval of $I$.  Then for all  $k \in \bZ^+$,
\begin{equation}
	\rho(M_{k+1}) \leq \rho (M_k).
\end{equation}
As a result, $\lim_{k\to \infty} \rho(M_k)$ exists, denoted by $\rho^*$.
\end{theo}
\begin{proof}
In order to prove the theorem, we introduce another $N^{k+1}\times N^{k+1}$ matrix $\wdtM_{k}$ as follows:
\begin{equation}\label{eq:Mktilde-def}
	(\widetilde{M}_{k})_{(i-1)N^{k}+\ell,j}=\begin{cases}
    \overline{s}^k_{i,\ell}, & \mbox{if } (\ell-1)N< j\leq \ell N, \\
    0, & \mbox{otherwise},
  \end{cases}
\end{equation}
for $1\leq i\leq N$, $1\leq \ell\leq N^k$ and $1\leq j\leq N^{k+1}$.


Firstly, we will prove
	$\rho \big( \wdtM_k \big) \geq \rho \big(M_{k+1}\big).$
Given $1\leq i \leq N$ and $1\leq \ell \leq  N^k$,
\begin{equation*}
  I_{i,\ell}^k=(I_i)^k_{\ell}=\bigcup_{j=(\ell-1)N+1}^{\ell N} (I_i)_{j}^{k+1}=\bigcup_{j=(\ell-1)N+1}^{\ell N} I_{i,j}^{k+1}
\end{equation*}
so that
\begin{equation*}
	\overline{s}_{i,\ell}^k
	=\max_{x \in I_{i,\ell}^k} |S(x)|
	= \max_{(\ell-1)N<j\leq \ell N }\max_{x \in I_{i,j}^{k+1}} |S(x)|=\max_{(\ell-1)N<j\leq \ell N } \overline{s}_{i,j}^{k+1}.
\end{equation*}
Thus, given $1\leq j\leq N^{k+1}$ with $(\ell-1)N<j\leq \ell N$, we have
$
 \overline{s}_{i,\ell}^k \geq  \overline{s}_{i,j}^{k+1}.
$
Combining this with definitions of $\wdtM_k$ and $M_{k+1}$, it is easy to see that $\wdtM_k \geq M_{k+1}$. From Perron-Frobenius Theorem, $\rho \big( \wdtM_k \big) \geq \rho \big(M_{k+1}\big)$.


Now we will prove $\rho \big(M_k\big)=\rho \big(\wdtM_k \big).$ From Perron-Frobenius Theorem,  the spectral radius $\lambda=\rho(\wdtM_k )$ is positive, and it has a  positive eigenvector $ u=(u_1,\ldots,u_{N^{k+1}})^T$. From \eqref{eq:Mktilde-def}, for $1\leq i\leq N$ and $1\leq \ell\leq N^k$,
\begin{equation}\label{eq:eigenv-Mktilde}
  \lambda u_{(i-1)N^k+\ell}=\sum_{j=1}^{N^{k+1}} (\widetilde{M}_k)_{(i-1)N^k+\ell, j} u_j = \overline{s}_{i,\ell}^k \sum_{j=(\ell-1)N+1}^{\ell N} u_{j}.
\end{equation}


 Define a vector $v=(v_1,\ldots,v_{N^k})^T$ by
$
 v_t=\sum_{j=(t-1)N+1}^{t N} u_{j},  1\leq t\leq N^k.
$ Then $v$ is a positive vector.
Notice that for $1\leq i\leq N$ and $1\leq \ell\leq N^{k-1}$,
\begin{align*}
	\sum_{t=1}^{N^k} (M_k)_{(i-1)N^{k-1}+\ell,t} v_{t}
	&=\sum_{t=(\ell-1)N+1}^{\ell N} \overline{s}_{i,t}^k  v_{t} \\
	&=\sum_{t=(\ell-1)N+1}^{\ell N}  \Big(\overline{s}_{i,t}^k \sum_{j=(t-1)N+1}^{tN}  u_{j}\Big) \\
	&=\sum_{t=(\ell-1)N+1}^{\ell N} \lambda  u_{(i-1)N^k+t} \qquad \qquad (\textrm{By \eqref{eq:eigenv-Mktilde}})\\
	&=\lambda  v_{(i-1)N^{k-1}+\ell}.
\end{align*}
Thus $M_k v=\lambda v$ so that $\lambda$ is an eigenvalue of $M_k$ with the positive eigenvector $v$. From \cite[Corollary 8.1.30]{HorJoh90},
$
	 \rho \big(M_k\big) = \lambda =\rho \big( \wdtM_k \big).
$

From the above arguments, $\rho \big(M_k\big) =\rho \big( \wdtM_k \big)\geq \rho \big(M_{k+1}\big).$ Since $\rho(M_k)>0$ for all $k$, we know that $\lim_{k\to\infty} \rho(M_k)$ exists.
\end{proof}

Similarly, we can obtain the following result.
\begin{theo}\label{th:rho-m}
Assume that the the vertical scaling function $S$ is positive on $I$. Then for all $k \in \mathbb{Z}^+$,
\begin{equation}
	\rho(M'_k) \leq \rho (M'_{k+1}).
\end{equation}
As a result, $\lim_{k\to \infty} \rho(M'_k)$ exists, denoted by $\rho_*$.
\end{theo}
\begin{proof}
From the assumption of this theorem, we have $\underline{s}^k_{i,j}>0$ for all $k\in \bZ^+$, $1\leq i \leq N$ and $1\leq j \leq N^k$.

Now, for each $k\in \bZ^+$, we define another $N^{k+1} \times N^{k+1}$ matrix $\widetilde{M}'_k$ as follows:
\begin{equation}\label{eq:mktilde-def}
	(\widetilde{M}'_k)_{(i-1)N^{k}+\ell,j}=\begin{cases}
    \underline{s}^k_{i,\ell}, & \mbox{if } (\ell-1)N< j\leq \ell N, \\
    0, & \mbox{otherwise},
  \end{cases}
\end{equation}
for $1\leq i\leq N$, $1\leq \ell\leq N^k$ and $1\leq j\leq N^{k+1}$. 

 By using similar arguments as in the proof of Theorem~\ref{th:rho-M}, we have
$	\rho(M'_{k+1}) \geq \rho (\widetilde{M}'_k) = \rho (M'_k) >0.$ Since $\rho(M'_k)\leq \rho(M_k)\leq \rho(M_1)$ for all $k$, we know that $\lim_{k\to \infty} \rho(M'_k)$ exists.
\end{proof}

\begin{prop}\label{prop:rho-rho}
	Assume that the the vertical scaling function $S$ is positive on $I$. Then $\rho_*=\rho^*.$ We denote the common value by $\rho_S$.
\end{prop}

\begin{proof}
Write $C=(\min\{S(x):\, x\in I\})^{-1}$. Then $0<C<\infty$.

From definitions of $M_k$ and $M'_k$, we have
\begin{equation*}
	0\leq \overline{s}_{i,j}^k-\underline{s}_{i,j}^k=\sup_{x \in I_{i,j}^k}|S(x)|-\inf_{x \in I_{i,j}^k} |S(x)|
	\leq \lambda_S |I_{i,j}^k| = \lambda_S N^{-k-1}
\end{equation*}
so that
\begin{align*}
	\underline{s}_{i,j}^k \leq \overline{s}_{i,j}^k
	\leq \Big(1 +\frac{\lambda_S N^{-k-1}}{\underline{s}_{i,j}^k}\Big)\underline{s}_{i,j}^k  \leq (1 +C\lambda_S N^{-k-1})\underline{s}_{i,j}^k.
\end{align*}
Thus $M'_k \leq M_k \leq (1 +C\lambda_S N^{-k-1}) M'_k$. Hence
\begin{equation*}
	\rho (M'_k) \leq \rho( M_k) \leq (1 +C\lambda_S N^{-k-1}) \rho(M'_k),
\end{equation*}
which implies that 	$\rho^*=\lim_{k \to \infty} \rho(M_k)=\lim_{k \to \infty} \rho(M'_k)=\rho_*.$
\end{proof}


\subsection{The sum function}

Now, we define a function $\gamma$ on $I$ by
\begin{equation*}
	\gamma(x)=\sum_{i=1}^{N} |S(L_i(x)|.
\end{equation*}
We call $\gamma$ the \emph{sum function} with respect to $S$ and $\{L_i\}_{i=1}^N$.
Write
\begin{equation*}
	\gamma^*=\max_{x \in I} \gamma(x), \qquad \quad \gamma_*=\min_{x \in I} \gamma(x).
\end{equation*}
For any $k\in \bZ^+$, we define
\begin{equation*}
	\overline{\gamma}_k=\max_{1 \leq j \leq N^k}
	\sum_{i=1}^{N} \overline{s}_{i,j}^k,\quad
	\underline{\gamma}_k=\min_{1 \leq j \leq N^k}
	\sum_{i=1}^{N} \underline{s}_{i,j}^k.
\end{equation*}

\begin{lem}\label{lem:lem01}
$\lim_{k \to \infty} \overline{\gamma}_k=\gamma^*$ and $\lim_{k \to \infty} \underline{\gamma}_k=\gamma_*.$
\end{lem}
\begin{proof}
Fix $k\in \bZ^+$. For any $\wdt{x}\in I$, there exists $1\leq j \leq N^k$ such that $\wdt{x} \in I_j^k$. Thus
\begin{equation*}
	\sum_{i=1}^{N} \overline{s}_{i,j}^k=\sum_{i=1}^{N} {\max_{x \in I_j^k} |S(L_i(x))|} \geq \sum\limits_{i=1}^{N}{|S(L_i(\wdt{x}))|}=\gamma (\wdt{x}),
\end{equation*}
which implies that $\overline{\gamma}_k \geq \gamma (\wdt{x})$. By the arbitrariness of $\wdt{x}$, we have $\overline{\gamma}_k \geq  \gamma^*$.

On the other hand, given $1\leq j \leq N^k$, for any $x' \in I_j^k$ and $1\leq i\leq N$, it follows from \eqref{eq:3-1} that
\begin{equation*}
	|S(L_i(x'))|\geq \max_{x \in I_{i,j}^k} |S(x)|- \lambda_S \cdot |I_{i,j}^k| =\overline{s}_{i,j}^k-\lambda_S N^{-k-1}.
\end{equation*}
Thus
\[
  \gamma^*\geq \gamma(x')=\sum_{i=1}^{N} |S(L_i(x'))|\geq \sum_{i=1}^{N} \overline{s}_{i,j}^k-\lambda_S N^{-k}.
\]
Since this inequality holds for all $1\leq j\leq N^k$, we obtain that $\gamma^*\geq \overline{\gamma}_k-\lambda_S N^{-k}$. Combining this with  $\overline{\gamma}_k \geq  \gamma^*$, we obtain that $\overline{\gamma}_k\geq   \gamma^* \geq \overline{\gamma}_k- \lambda_S N^{-k}$ for all $k\in \bZ^+$.
Thus $\lim_{k \to \infty}  \overline{\gamma}_k = \gamma^*$.

Similarly, we have $\underline{\gamma}_k\leq \gamma_* \leq \underline{\gamma}_k+\lambda_S N^{-k}$ for all $k$ so that $\lim_{k \to \infty} \underline{\gamma}_k=\gamma_*$.
\end{proof}


Notice that $M_k$ and $M'_k$ are nonnegative matrices for every $k\in \bZ^+$. Thus, from \cite[Theorem~8.1.22]{HorJoh90},
\begin{align*}	
	\rho(M_k) \leq \max_{1\leq j\leq N^k}\sum_{i=1}^{N^k} (M_k)_{i,j} =\overline{\gamma}_k,\qquad
	\rho(M'_k) \geq \min_{1\leq j\leq N^k}\sum_{i=1}^{N^k} (M'_k)_{i,j} =\underline{\gamma}_k.
\end{align*}
Hence, if $S$ is not identically zero on every subinterval of $I$, then $\rho^*\leq\gamma^*$; and if $S$ is  positive on $I$, then $\rho_*\geq\gamma_*$. From Proposition~\ref{prop:rho-rho}, if $S$ is positive, then $\gamma_*\leq \rho_S\leq \gamma^*$. In particular, if $S$ is positive and $\gamma$ is constant on $I$, then $\gamma(x)=\rho_S$ for all $x\in I$.

\begin{rem}\label{rem:bilinear}
In \cite{BaMa15}, Barnsley and Massopust study the box dimension of the graph of bilinear FIFs. In their setting, the vertical scaling function $S$ is nonnegative and linear on $I_i=[\frac{i-1}{N},\frac{i}{N}]$ for each $1\leq i\leq N$. Furthermore, they assume that $S(0)=S(1)$. Denote by $s_i=S(i/N)$ for all $0\leq i\leq N$. From the above two conditions, it is easy to check that for all $x\in I$,
\[
  \gamma(x)=\sum_{i=1}^N |S(L_i(x))|=\sum_{i=1}^N ((1-x)s_{i-1} + x s_i)=\sum_{i=1}^N s_i.
\]
\end{rem}

\section{Calculation of box dimension of generalized affine FIFs}

In this section, we always assume that $f$ is a generalized affine FIF on $I=[0,1]$ satisfying the conditions (A1)-(A5). We will estimate the lower box dimension of $\Gamma f$ under another condition (A6), and estimate the upper box dimension of $\Gamma f$ under the following weaker condition:
\begin{itemize}
  \item[(A6')] The vertical scaling function $S$ is not identically zero on every subset of $I$.
\end{itemize}
Combining these two results, we obtain the formula of $\Bdim \Gamma f$ under the conditions (A1)-(A6).
As explained in Section~2, we may assume without loss of generality that $y_0=y_N=0$ so that $b(x)\equiv 0$ on $I$.





\subsection{Box dimension of the graph of functions}


Let $g$ be a continuous function on $I=[0,1]$.
Given $k\in \mathbb{Z}^+$ and a closed interval $E\subset I$, we define
\begin{equation}
	O_k(g,E)=\sum\limits_{j=1}^{N^k} O(g,E_j^k),
\end{equation}
where we use $O(g,U)$ to denote the oscillation of $g$ on $U\subset I$, that is,
\begin{equation*}
	O(g,U)= \sup\limits_{x',x'' \in U}|g({x}')-g({x}'')|.
\end{equation*}
It is clear that $\{O_k(g,E)\}_{k=1}^\infty$ is increasing with respect to $k$.

The following lemma presents a method to estimate the upper and lower box dimensions of the graph of a function by its oscillation. Similar results can be found in \cite{Fal90,KRZ18,RSY09}.

\begin{lem}\label{lem:box-dim-new}
Let $g$ be a continuous function on $I$. Then
\begin{equation}\label{eq:lower-boxdim-est-new}
	\underline{\dim}_B(\Gamma g) \geq 1+\varliminf_{k\to\infty}\frac{\log \big( O_k(g,I)+1\big)}{k\log N}, 	\quad \mbox{and}
\end{equation}
\begin{equation}\label{eq:upper-boxdim-est-new}
  	\overline{\dim}_B(\Gamma g) \leq 1+\varlimsup_{k\to\infty} \frac{\log \big(O_k(g,I)+1\big)}{k\log N}.
\end{equation}
\end{lem}

\begin{proof}
Define $\varepsilon_k=N^{-k}$ for all $k\in \bZ^+$ as in  Subsection~2.1.
It is clear that
\begin{equation*}
	\N_{\Gamma g}(\varepsilon_k) \geq \max\Big\{\gep_k^{-1},\varepsilon_k^{-1} \sum_{1\leq j\leq N^k} O(g,I_j^k)\Big\}\geq \frac{1}{2}N^k  \big(1+O_k(g,I)\big)
\end{equation*}
so that \eqref{eq:lower-boxdim-est-new} holds. On the other hand, we note that $\N_E(\gep)$ and $\N_E(\gep_k)$ can be replaced by $\wdt{\N}_E(\gep)$ and $\wdt{\N}_E(\gep_k)$ in \eqref{eq:box-dim-def} and \eqref{eq:box-dim-def2} respectively, where $\wdt{N}_E(\gep)$ is the smallest number of cubes of side $\gep$ that cover $E$ (see \cite{Fal90} for details). In our case, 
\begin{equation*}
	\widetilde{\N}_{\Gamma g} (\varepsilon_k)  \leq \sum_{1\leq j\leq N^k} (\varepsilon_k^{-1}O(g,I_j^k)+1)
	=  N^k  \big(O_k(g,I)+1\big).
\end{equation*}
Thus \eqref{eq:upper-boxdim-est-new} holds.
\end{proof}



\subsection{Upper box dimension estimate}
In this subsection, we will estimate the upper bound of $O_k(f,I)$ and the upper box dimension of $\Gamma f$.



\begin{lem}\label{lem:of-ofl}
There exists a constant $\beta\geq 0$ such that for any $1\leq i\leq N$, $D \subset I_i$ and any $t \in D$,
\begin{equation*}	
  \big| O(f,D) -|S(t)| \cdot O(f,L_i^{-1} (D)) \big|\leq \beta|D|,
 \end{equation*}
where $|D|=\sup\{|x'-x''|:\, x',x''\in D\}$ is the diameter of $D$.
\end{lem}
\begin{proof}
Given $ D \subset I_i$ and $t\in D$,  from \eqref{eq:FIF-rec},
\begin{equation*}
	O(f,D)=\sup_{x',x'' \in D}\big|h(x')-h(x'')+S(x')f(L_i^{-1}(x'))-S(x'')f(L_i^{-1}(x''))\big|.
\end{equation*}

Write $M_f=\max_{x \in I}|f(x)|$. Notice that for any $x',x''\in D$,
\[
  |h(x')-h(x'')|=|x'-x''|\cdot|N(y_i-y_{i-1})|\leq |D|\cdot \max_{1\leq \ell\leq N} |N(y_\ell-y_{\ell-1})|,
\]
and
\begin{align*}
  &|S(x')f(L_i^{-1}(x'))-S(x'')f(L_i^{-1}(x''))| \\
  \leq & |S(x')-S(t)|\cdot |f(L_i^{-1}(x'))| + |S(x'')-S(t)|\cdot |f(L_i^{-1}(x''))| \\
        &+ |S(t)|\cdot |f(L_i^{-1}(x'))-f(L_i^{-1}(x''))| \\
    \leq & 2 M_f \lambda_S |D|+ |S(t)|\cdot O(f,L_i^{-1}(D)).
\end{align*}
Let $\beta=2 M_f \lambda_S + \max_{1\leq \ell\leq N} |N(y_\ell-y_{\ell-1})|$. Then $\beta\geq 0$ and
\[
  O(f,D) \leq |S(t)|\cdot O(f,L_i^{-1}(D))+\beta |D|.
\]
Similarly, it is easy to see that
\[
  O(f,D) \geq |S(t)|\cdot O(f,L_i^{-1}(D))-\beta|D|.
\]
Thus, the lemma holds.
\end{proof}

\begin{cor}\label{cor:1}
Let $\beta\geq 0$ be the constant in Lemma~\ref{lem:of-ofl}. Then for any $1\leq i \leq N$, $k \in \mathbb{Z}^+$, $1\leq j \leq N^k$ and any $t \in I_{i,j}^k$,
\begin{align*}
	\big|O(f,I_{i,j}^k)-|S(t)| O(f,I_j^k) \big| \leq \beta N^{-k-1}.
\end{align*}
\end{cor}

\begin{cor}\label{cor:2}
Let $\beta\geq 0$ be the constant in Lemma~\ref{lem:of-ofl}. Then for any $1\leq i \leq N$, $k \in \mathbb{Z}^+$, $1\leq j \leq N^k$ and any $D \subset I_{i,j}^k$,
\begin{equation*}
	\underline{s}_{i,j}^{k} \cdot O(f,L_i^{-1} (D)) -\beta|D|
	\leq O(f,D) \leq \overline{s}_{i,j}^{k} \cdot O(f,L_i^{-1}{(D)}) + \beta|D|.
\end{equation*}
\end{cor}

Given $k,q \in \mathbb{Z}^+$, we define
\begin{equation*}
	V(f,k,q)=\big(O_q(f,I_1^k),O_q(f,I_2^k),\cdots,O_q(f,I_{N^k}^k)\big)^T \in \mathbb{R}^{N^k},
\end{equation*}
and
\begin{equation*}
	V(f,k)=\big(O(f,I_1^k),O(f,I_2^k),\cdots,O(f,I_{N^k}^k)\big)^T \in \mathbb{R}^{N^k}.
\end{equation*}
For convenience, we write $V(f,k,0)=V(f,k)$.
It is obvious that for all $k\in \mathbb{Z}^+$ and $q\in \mathbb{N}$,
\begin{equation*}
	O_{k+q}(f,I)={||V(f,k+q)||}_1={||{V}(f,k,q)||}_1,
\end{equation*}
where $||v||_1:=\sum_{i=1}^n |v_i|$ for any $v=(v_1,\ldots,v_n)\in \bR^n$.


\begin{lem}\label{lem:vpq}
Assume that the function $S$ is not identically zero on every subinterval of $I$.  Then for every  $k \in \bZ^+$,
there exists a constant $c_k>0$, such that for all $q \in \bZ^+$,
\begin{equation*}
O_{k+q}(f,I) \leq
\begin{cases}
	c_k  (\rho (M_k) )^q, \quad	& \rho(M_k) >1,	\\
	c_k   q,  \quad &\rho(M_k) \leq 1.
\end{cases}
\end{equation*}
\end{lem}
\begin{proof}
Let $\beta\geq 0$ be the constant in Lemma~\ref{lem:of-ofl}.
Given $q \in \mathbb{Z}^+$, $1\leq i \leq N$ and $1\leq j \leq N^{k-1}$, from Corollary~\ref{cor:2},
\begin{align*}
  O_q(f,I_{i,j}^k)&=\sum_{m=1}^{N^q} O(f,(I_{i,j}^k)_m^q) \\
  &\leq \sum_{m=1}^{N^q}\Big( \overline{s}_{i,j}^k O(f,L_i^{-1}((I_{i,j}^k)_m^q)) +\beta |(I_{i,j}^k)_m^q| \Big)\\
  &=\overline{s}_{i,j}^k \sum_{m=1}^{N^q} O(f,(I_j^k)_m^q) + \beta |I_{i,j}^k|
  =\overline{s}_{i,j}^k O_q(f,I_j^k) + \beta N^{-k-1}.
\end{align*}
Thus, for $1\leq i\leq N$ and $1\leq \ell \leq N^{k-1}$, from 
\[
  I_{(i-1)N^{k-1}+\ell}^k = I_{i,\ell}^{k-1}=\bigcup_{j=(\ell-1)N+1}^{\ell N} I_{i,j}^k,
\]
we have
\begin{align*}
  O_{q+1}(f,I_{(i-1)N^{k-1}+\ell}^k)
  &=\sum_{j=(\ell-1)N+1}^{\ell N} O_q(f,I_{i,j}^k) \\
  &\leq \beta N^{-k}+\sum_{j=(\ell-1)N+1}^{\ell N} \overline{s}_{i,j}^k O_q(f,I_j^k). 
\end{align*}
Denote the vector $(\beta N^{-k}, \ldots, \beta N^{-k})^T$ in $\mathbb{R}^{N^k}$ by $u$. From the above inequality,
\begin{align*}
	V(f,k,q+1)\leq u +M_k V(f,k,q),
\end{align*}
which implies that
\begin{gather*}
	V(f,k,q) \leq u+M_k u+\cdots+(M_k)^{q-2}u+(M_k)^{q-1} V(f,k,1)
\end{gather*}
for all $q\in \bZ^+$.
By Lemma~\ref{th:PF}, we can choose a strictly positive eigenvector $v$ of $M_k$ such that $v\geq V(f,k,1)$ and $v\geq u$. Thus
\begin{align*}
	&(M_k)^{q-1}  V(f,k,1) \leq   (M_k)^{q-1} v = (\rho(M_k))^{q-1} v, \quad  \mbox{and} \\
	&(M_k)^n u \leq  (M_k)^n v = (\rho(M_k))^n v, \quad n\in \mathbb{N}.
\end{align*}
Hence
\begin{equation*}
  O_{k+q}(f,I)=||V(f,k,q)||_1\leq ||v||_1 \sum_{n=0}^{q-1} (\rho(M_k))^n.
\end{equation*}

Let $c_k=||v||_1$ if $\rho(M_k)\leq 1$, and $c_k=||v||_1/(\rho(M_k)-1)$ if $\rho(M_k)>1$. Then the lemma holds.
\end{proof}

\begin{theo}\label{th:upper}
Assume that the function $S$ is not identically zero on every subinterval of $I$.
Then
\begin{equation}\label{eq:upBoxFormula}
	\overline{\dim}_{B} \Gamma f  \leq  \max\Big\{1,1+\frac{\log  \rho^*   }{\log N}\Big\}.
\end{equation}

\end{theo}
\begin{proof}

From Lemma~\ref{lem:box-dim-new}, for every fixed $k\in \bZ^+$,
\begin{equation}\label{eq:th-upper-1}
    \overline{\dim}_B \Gamma f \leq 1+\varlimsup_{q\to\infty} \frac{\log (O_{k+q}(f,I)+1)}{q\log N}.
\end{equation}

Assume that $\rho(M_k)\leq 1$ for some $k\in \bZ^+$. From Lemma~\ref{lem:vpq},  there exists a constant $c_k>0$, such that
\begin{equation*}
	O_{k+q}(f,I) \leq c_k q
\end{equation*}
for all $q \in \bZ^+$. Combining this with \eqref{eq:th-upper-1}, we have $\upBdim \Gamma f\leq 1$.  Thus \eqref{eq:upBoxFormula} holds.

Assume that $\rho(M_k)> 1$ for all $k\in \bZ^+$. From Lemma~\ref{lem:vpq}, for each $k\in \bZ^+$, there exists a constant $c_k>0$, such that
\begin{equation*}
	O_{k+q}(f,I) \leq c_k  {\rho (M_k) }^q
\end{equation*}
for all $q \in \bZ^+$. Combining this with \eqref{eq:th-upper-1}, we have $\upBdim \Gamma f \leq 1+\frac{\log \rho(M_k)}{\log N}$. By the arbitrariness of $k$, we know from Theorem~\ref{th:rho-M} that
\[
  \upBdim \Gamma f\leq 1+\lim_{k\to \infty} \frac{\log \rho(M_k)}{\log N}=1+\frac{\log \rho^*}{\log N}.
\]
Thus \eqref{eq:upBoxFormula} also holds.
\end{proof}

\subsection{Lower box dimension estimate}
In this subsection, we will estimate the lower bound of $O_k(f,I)$ and the lower box dimension of $\Gamma f$.


\begin{lem}\label{lem:limit_o}
Assume that the function $S$ is positive.
If $\lim_{p\to \infty} O_p(f,I)=\infty$, then for any $k \in \bZ^+$ and any positive vector $v \in \mathbb{R}^{N^k}$, there exists $p\in  \bZ^+$ satisfying
\begin{align}
	 V(f,k,p) \geq v.
\end{align}
\end{lem}
\begin{proof}
Let $\beta\geq 0$ be the constant defined in Lemma~\ref{lem:of-ofl}. Fix $k\in \bZ^+$.
Given $q \in \bZ^+$, $1\leq i \leq N$ and $1\leq \ell \leq N^{k-1}$, using similar arguments in Lemma~\ref{lem:vpq}, we have
\begin{align*}
  O_{q+1}(f,I_{(i-1)N^{k-1}+\ell}^k)
  \geq \sum_{j=(\ell-1)N+1}^{\ell N} \overline{s}_{i,j}^k O_q(f,I_j^k)-\beta N^{-k}
\end{align*}
so that
\begin{equation}\label{eq:lower-connect}
	V(f,k,q+1)\geq M'_k V(f,k,q) - u,
\end{equation}
where $u=(\beta N^{-k}, \ldots, \beta N^{-k})^T \in \mathbb{R}^{N^k}$. By induction,
\begin{equation}\label{eq:lem:limit_o-1}
	V(f,k,q+k) \geq (M'_k)^k V(f,k,q)-\sum_{\ell=0}^{k-1} (M'_k)^\ell u.
\end{equation}
Let $\alpha_k$ be the minimal entry of the matrix $(M'_k)^k$. Then $\alpha_k>0$ since $(M'_k)^k>0$. Notice that the maximal element of $V(f,k,q)$ is at least $N^{-k}||V(f,k,q)||_1$. Thus
\begin{equation}\label{eq:lem:limit_o-2}
  (M'_k)^k V(f,k,q)\geq ||V(f,k,q)||_1 w_k,
\end{equation}
where $w_k=(\alpha_kN^{-k},\ldots,\alpha_kN^{-k})\in \bR^{N^k}$.



On the other hand, it follows from $\lim_{p\to \infty} O_p(f,I)=\infty$ that
\begin{equation*}
 \lim_{q\to \infty}{||V(f,k,q)||}_1=\lim_{q\to \infty}O_{k+q}(f,I)=\infty.
\end{equation*}
Hence, we can choose $q_0$ large enough, such that
\[
  ||V(f,k,q_0)||_1w_k\geq v+\sum_{\ell=0}^{k-1}(M'_k)^\ell u.
\]
Combining this with \eqref{eq:lem:limit_o-1} and \eqref{eq:lem:limit_o-2}, we have $V(f,k,p) \geq v$, where $p=q_0+k$.
\end{proof}

\begin{theo}\label{th:lower}

Assume that the function $S$ is positive and $\lim_{p\to \infty} O_p(f,I)=\infty$. Then
\begin{align}\label{eq:thm-lowBdim}
	\lowBdim \Gamma f
	\geq 1+\frac{\log \rho_S}{\log N} .
\end{align}


\end{theo}
\begin{proof}
Fix $k\in \bZ^+$.
Let $\beta \geq 0$ be the constant defined in Lemma~\ref{lem:of-ofl} and
\[
  u=(\beta N^{-k}, \ldots, \beta N^{-k})^T \in \mathbb{R}^{N^k}.
\]


Given $0<\tau<\rho(M'_k)$, from Perron-Frobenius Theorem, we can choose a positive eigenvector $v$ of $M'_k$ associated with the eigenvalue $\rho(M'_k)$, such that
$
	v \geq \frac{1}{\rho(M'_k)-\tau} u.
$
Since $\lim_{p \to \infty}{O_p(f,I)}=\infty$, from Lemma~~\ref{lem:limit_o}, there exists $q\in \bZ^+$, such that
\begin{equation*}
	V(f,k,q) \geq v \geq \frac{1}{\rho(M'_k)-\tau} u.
\end{equation*}
Hence from \eqref{eq:lower-connect},
\[
  V(f,k,q+1)\geq\rho(M'_k)v-u\geq \rho(M'_k)v-(\rho(M'_k)-\tau)v=\tau v.
\]
Notice that all $n\in \bZ^+$,
\begin{align*}
  \rho(M'_k)\tau^n v-u &= \rho(M'_k)\big(\tau^n-1\big)v+\rho(M'_k)v-u \\
  &\geq \tau\big(\tau^{n}-1\big)v+\tau v=\tau^{n+1} v.
\end{align*}
Thus, by induction, for all $n\in \bZ^+$,
\begin{equation*}
	V(f,k,q+n) \geq \tau^{n}  v.
\end{equation*}
Hence
$
	O_{k+q+n}(f,I) = {||V(f,k,q+n)||}_1 \geq \tau^{n} {||  v ||}_1,
$ which implies that
\begin{align*}
	\varliminf_{n \to \infty}{\frac{\log \big(O_{n}(f,I)+1\big)}{n \log N}}
	&=	\varliminf_{n \to \infty}{\frac{\log \big(O_{k+q+n}(f,I) +1\big)}{(k+q+n) \log N}} \\
	&\geq \varliminf_{n \to \infty}{\frac{\log \big(\tau^n {||  v ||}_1 +1 \big)}{n \log N}}
	\geq \frac{\log \tau}{\log N} .
\end{align*}


From the arbitrariness of $\tau$,
\begin{equation*}
	\varliminf_{n \to \infty}{\frac{\log \big(O_{n}(f,I)+1\big)}{n \log N}}
	\geq  \frac{\log \big( \rho(M'_k) \big)}{\log N} .
\end{equation*}
By letting $k\to \infty$, we know from Theorem~\ref{th:rho-m} and Proposition~\ref{prop:rho-rho} that
\begin{equation*}
	\varliminf_{n \to \infty}{\frac{\log \big(O_{n}(f,I)+1\big)}{n \log N}}
	\geq  \frac{\log \rho_S} {\log N}.
\end{equation*}
Combining this with Lemma~\ref{lem:box-dim-new}, \eqref{eq:thm-lowBdim} holds.
\end{proof}

\subsection{The box dimension formula}

Notice that $\lowBdim \Gamma g \geq 1$ for every continuous function $g$ on $I$. Thus, from Theorems~\ref{th:upper} and \ref{th:lower}, the following result holds.
\begin{theo}\label{thm-bdim}
Assume that the function $S$ is positive on $I$. Then in the case that $\lim_{p\to \infty} O_p(f,I)=\infty$ and $\rho_S>1$,
\begin{equation}\label{eq:thm-Bdim}
	 \Bdim \Gamma f=1+\frac{\log  \rho_S   }{\log N}.
\end{equation}
Otherwise, $\Bdim  \Gamma f=1$.
\end{theo}
\begin{proof}
Since $S$ is positive on $I$, we know from Proposition~\ref{prop:rho-rho} that $\rho^*=\rho_*=\rho_S$.

Notice that the sequence $\{O_p(f,I)\}_{p=1}^\infty$ is increasing with respect to $p$. Thus the limit $\lim_{p\to \infty} O_p(f,I)$ exists.


In the case that $\lim_{p\to \infty} O_p(f,I)<\infty$, we know from Lemma~\ref{lem:box-dim-new} that $\upBdim \Gamma f\leq 1$.
In the case that $\rho_S\leq 1$, we know from Theorem~\ref{th:upper} that $\upBdim \Gamma f\leq 1$. Since $\lowBdim \Gamma f \geq 1$ always holds, $\Bdim  \Gamma f=1$ if $\lim_{p\to \infty} O_p(f,I)<\infty$ or $\rho_S\leq 1$.

In the case that $\lim_{p\to \infty} O_p(f,I)=\infty$ and $\rho_S>1$, we know from Theorems~\ref{th:upper} and \ref{th:lower} that \eqref{eq:thm-Bdim} holds.
\end{proof}


The following theorem gives us an easy-checking sufficient condition such that \eqref{eq:thm-Bdim} holds.	
\begin{theo}\label{thm:sufficient}

Assume that $\gamma_*>1$ and the function $S$ is positive on $I$. Furthermore, assume that there exists $k_0 \in \mathbb{Z}^+$ satisfying $O_{k_0}(f,I) >  \lambda_S M_f  / (\gamma_*-1)$, where $M_f =\max_{x \in I} |f(x)|$. Then \eqref{eq:thm-Bdim} holds.
\end{theo}
\begin{proof}Since the function $S$ is positive, $\gamma(x)=\sum_{i=1}^N S(L_i(x))$. It is easy to see that for any $x',x''\in I$,
\[
  |\gamma(x')-\gamma(x'')|\leq  \lambda_S \sum_{i=1}^N |L_i(x')-L_i(x'')|=\lambda_S |x'-x''|.
\]

From \eqref{eq:2-1} and noticing that $b(x)\equiv 0$ on $I$,
\begin{align*}
	\sum_{i=1}^N f(L_i(x))
	&=\sum_{i=1}^N  \big(S(L_i(x))f(x)+h(L_i(x)\big)		\\
	&=\gamma(x)f(x)+\sum_{i=1}^N \big((1-x)y_{i-1}+x y_i\big)		=\gamma(x)f(x)+\sum_{i=1}^N  y_i.		
\end{align*}
Hence
\begin{align*}
	\sum_{i=1}^N O(f,I_{i,j}^k)
	&=\sum_{i=1}^N \max_{x',x'' \in I_j^k} \big( f(L_i(x'))-f(L_i(x''))	 \big)	\\	
	&\geq \max_{x',x'' \in I_j^k} \big( \gamma(x')f(x')-\gamma(x'')f(x'')	\big)\\
	&=\max_{x',x'' \in I_j^k} \big( \gamma(x')(f(x')-f(x''))+(\gamma(x')-\gamma(x''))f(x'')	\big)\\
	&\geq \gamma_*  O(f,I_j^k) - \lambda_S  N^{-k}  M_f.
\end{align*}
Thus
\begin{align*}
	O_{k+1}(f,I)
	\geq \sum_{j=1}^{N^k} \big( \gamma_* O(f,I_j^k)-\lambda_S N^{-k} M_f \big)
	= \gamma_* O_k(f,I) - \lambda_S  M_f,
\end{align*}
which implies that
\[
  O_{k+1}(f,I)-c\geq \gamma_* (O_k(f,I)-c), \qquad k\in \bZ^+,
\]
where  $c=\gl_S M_f/(\gamma_*-1)$. Thus from $O_{k_0}(f,I)>c$ and $\gamma_*>1$, we have $\lim_{k\to \infty}O_k(f,I)=\infty$. Since  $\rho_S \geq \gamma_*>1$,
we know from Theorem~\ref{thm-bdim} that \eqref{eq:thm-Bdim} holds.
\end{proof}

\begin{rem}\label{rem:5-2}
Let $\lambda'$ be a Lipschitz constant of $\gamma$, i.e., $\lambda'$ is a positive constant satisfying
\[
  |\gamma(x')-\gamma(x'')|\leq  \lambda' |x'-x''|, \quad x',x''\in I.
\]
From the proof of the above theorem, if we replace the constant $\lambda_S$ by $\lambda'$, then the theorem still holds.
\end{rem}

As we mentioned in Section~3, in the case that $\gamma$ is positive and constant on $I$, we have $\gamma(x)=\rho_S$ for all $x\in I$. Thus, from the above remark, we have the following result.

\begin{cor}\label{cor:4-12}
Assume that the vertical scaling function $S$ is positive on $I$, and $\gamma(x) \equiv \gamma_0$ for all $x \in I$. Then in the case that $\gamma_0>1$ and the interpolation points are not collinear, $\Bdim \Gamma f=1+\log \gamma_0/\log N$. Otherwise, $\Bdim \Gamma f=1$.
\end{cor}

\begin{rem}\label{rem:4-13}
In \cite{BaMa15}, Barnsley and Massopust obtain the above result under a weaker condition: the condition ``$S$ is positive" was replaced by ``$S$ is nonnegative".

In \cite{KRZ18}, a so-called \emph{steady} condition was introduced in order to obtain the box dimension of bilinear fractal interpolation surfaces. In one dimensional case, we can define similarly. The vertical scaling function $S$ is called \emph{steady} on $I$ if for each $1\leq i\leq N$, either $S(x)\geq 0$ holds for all $x\in I_i$, or $S(x)\leq 0$ holds for $x\in I_i$. Assume that $\gamma(x)\equiv \gamma_0$ for all $x\in I$. By using the method in \cite{KRZ18}, if we replace the condition ``$S$ is positive on $I$" by ``$S$ is steady on $I$", the above corollary still holds.
\end{rem}








\section{An example and further remarks}

\subsection{An example}

\begin{exam}\label{exam-1}
Let $N=3$, $x_i=i/3$ for $i \in \{0,1,2,3\}$, $y_0=y_3=0$, and $y_1=y_2=1$.
Define
\begin{equation}\label{eq:exam-1}
	S(x)=\begin{cases}
	4/9,  &  		\quad	0			\leq x \leq \frac{1}{3},  \\
	x^2+1/3, &   		\quad	\frac{1}{3}	\leq x \leq \frac{2}{3}, \\
	13/9-x,  &  	\quad	\frac{2}{3}	\leq x \leq 1.	
	\end{cases}
\end{equation}
Let $f$ be the corresponding generalized affine FIF. See Figure~\ref{fig:FIF-Exam-1} for the graph of $f$.

Notice that for each $x\in I$, there exists a sequence $\{i_n\}_{n=1}^\infty$ with $i_n\in \{1,\ldots,N\}$ for each $n$, such that
$
  x\in \bigcap_{n=1}^\infty L_{i_1}\circ L_{i_2}\circ \cdots \circ L_{i_n}(I).
$
Thus, from \eqref{eq:FIF-rec},
\begin{align*}
  f(x)&=h(x)+S(x)f(L_{i_1}^{-1}(x)) \\
      &=h(x)+S(x)h(L_{i_1}^{-1}(x)) + S(x)S(L_{i_1}^{-1}(x))f\big(L_{i_2}^{-1}\circ L_{i_1}^{-1}(x)\big)\\
      &=h(x)+\sum_{n=1}^\infty \Big( S(x) S(L_{i_1}^{-1}(x))\cdots S\big(L_{i_{n-1}}^{-1} \circ \cdots \circ L_{i_1}^{-1}(x)\big) \Big) h\big( L_{i_n}^{-1}\circ \cdots L_{i_1}^{-1}(x)\big).
\end{align*}
It is easy to see that $|h(x)|\leq 1$ and $|S(x)|\leq 7/9$ for all $x\in I$. Thus
\[
  M_f=\max\{|f(x)|: x\in I\}\leq 1+\sum_{n=1}^\infty \Big(\frac{7}{9}\Big)^n=\frac{9}{2}.
\]

From \eqref{eq:exam-1},
\begin{equation}\label{eq:eg-lip}
	\gamma(x)=\sum_{i=1}^3 |S(L_i(x))|= \frac{x^2}{9}-\frac{x}{9} +\frac{5}{3}.
\end{equation}
Hence, $\gamma_*=59/36$, $\gamma^*=5/3$ and $\lambda'=1/9$ is a Lipschitz constant of $\gamma$. Thus $O_1(f,I)=2>\lambda' M_f/(\gamma_*-1)$. From Remark~\ref{rem:5-2}, \eqref{eq:thm-Bdim} holds.

By definition,

\begin{equation*}
M_1=
\left(
\begin{array}{ccc}
4/9	 	& 4/9		& 4/9		\\
43/81  	& 	52/81	& 7/9 		\\
7/9 	&2/3 		&  5/9
\end{array}
\right)
\quad \mbox{and} \quad
M'_1=
\left(
\begin{array}{ccc}
4/9	 	& 4/9		& 4/9		\\
4/9  	& 	43/81	& 52/81 	\\
2/3 	&5/9		&  4/9
\end{array}
\right).
\end{equation*}

In general, by calculation, we can obtain the spectral radii of vertical scaling matrices $\rho(M_k)$ and $\rho(M'_k)$, $1\leq k\leq 8$ as in Tabel~\ref{table:rho}.
Thus,
\[
  \dim_B \Gamma f = 1+\log \rho_S / \log N  \approx 1 + \log 1.647 /\log 3 \approx 1.454.
\]

\begin{table}[htbp]
	\centering
	\begin{tabular}{|c|c|c|c|c|c|c|c|c|}
	\hline 												
	$k$ &  1	&  2 	&  3 & 4 & 5 & 6 & 7 & 8  \\	\hline
	$\rho(M_k)$ 	&	1.7622	&  1.6852		& 1.6599	& 1.6515 & 1.6488 & 1.6478 &  1.6475 &  1.6474	\\	\hline
	$\rho(M'_k)$ 	&	1.5380	&	1.6102	&	1.6349 & 1.6432 & 1.6460 &  1.6469 & 1.6472 & 1.6473	\\	\hline
	\end{tabular}
	\caption{$\rho(M_k)$ and $\rho(M_k')$ in Example~\ref{exam-1}}
	\label{table:rho}
\end{table}

\end{exam}

\begin{figure}[htbp]
    \centering
    \includegraphics[width=5cm]{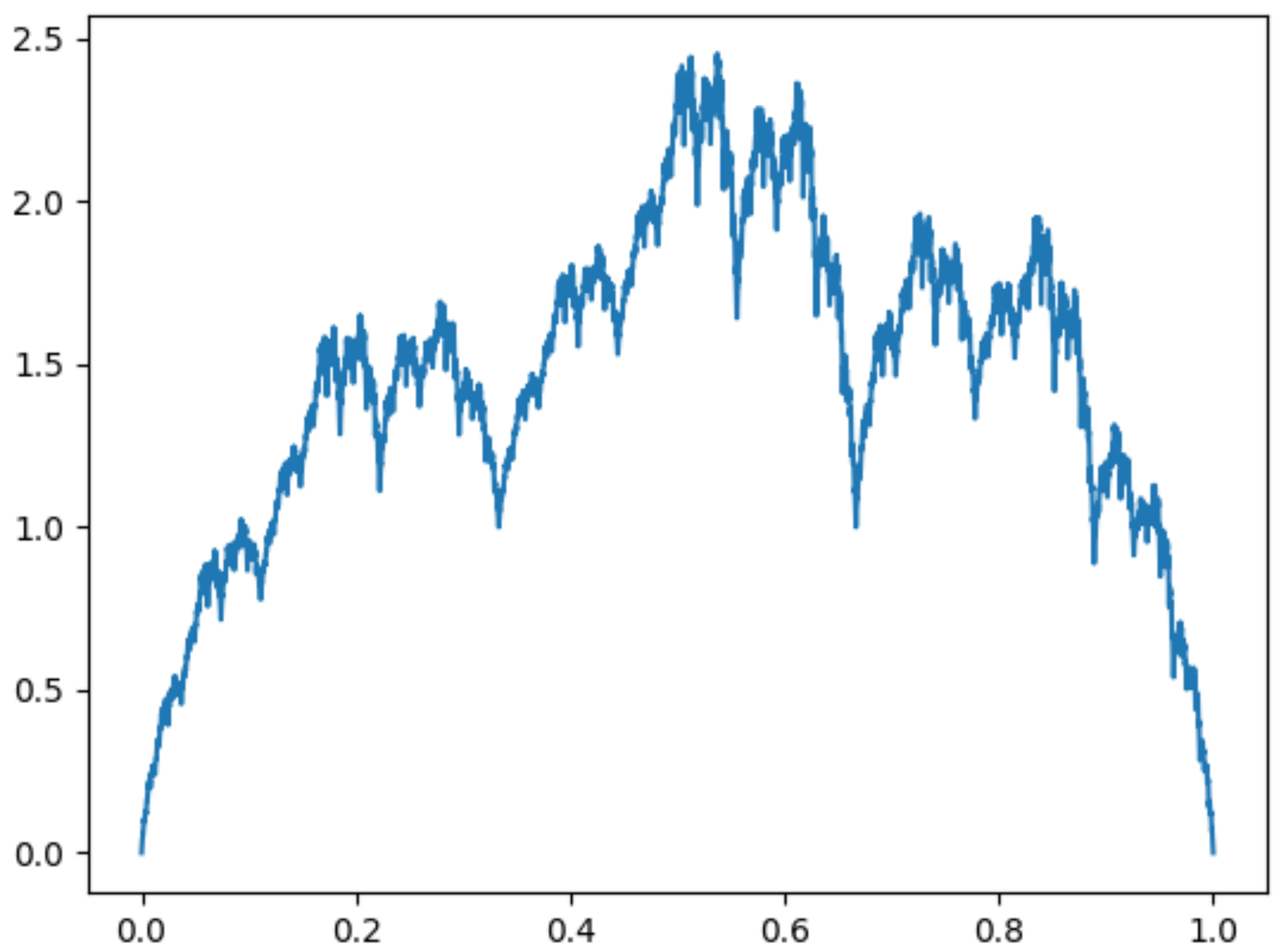}
    \caption{The FIF in Example~\ref{exam-1}}
    \label{fig:FIF-Exam-1}
\end{figure}

\begin{rem}
  Assume that the vertical scaling function $S$ is positive on $I$.
  Notice that 
  \begin{align*}	
	\rho(M'_k)\leq \rho(M_k) \leq \max_{1\leq j\leq N^k}\sum_{i=1}^{N^k} (M_k)_{i,j}<N.
  \end{align*}
  Hence, from the proof of Proposition~\ref{prop:rho-rho},
  \[
    \rho(M'_k)\leq \rho(M_k)\leq (1+C\lambda_S N^{-k-1}) \rho(M'_k)<\rho(M'_k)+C\lambda_S N^{-k},
  \]
  where $C=(\min\{S(x):\, x\in I\})^{-1}$. Thus
  \[
    0\leq \rho(M_k)-\rho(M'_k)<C\lambda_S N^{-k}.
  \]
  If for any $k\geq 1$, $1\leq i\leq N$ and $1\leq j\leq N^k$, we arbitrarily pick $x_{i,j}^k\in I_{i,j}^k$ and define an $N^k\times N^k$ matrix $T_k$ by replacing $\overline{s}_{i,j}^k$ with $S(x_{i,j}^k)$ in \eqref{eq:Mk-def}. Then
  \[
    |\rho_S-\rho(T_k)| \leq \rho(M_k) - \rho(M'_k)<C\lambda_S N^{-k}.
  \]
  
  For example, we can define $x_{i,j}^k$ to be the left endpoint (or right endpoint) of $I_{i,j}^k$. This gives us an effective method to estimate $\rho_S$.
\end{rem}

\subsection{Further remarks}


Recently, there are many deep works on Hausdorff dimension of self-affine sets and classical fractal functions, see \cite{BHR19,FalKem16,RenShen21} and the references therein. In particular,
 B\'{a}r\'{a}ny, Rams and Simon \cite{BRS20} proved that under certain conditions, the Hausdorff dimension and the box dimension of recurrent FIFs coincide. Thus, it is quite natural to see whether methods in these papers are applicable to our setting.


In the end of this paper, we pose some questions and conjectures related with our work.

\begin{conj}
Assume that the the vertical scaling function $S$ is not identically zero on every subinterval of $I=[0,1]$. Then
$\lim_{k\to \infty}\rho(M_k)=\lim_{k\to \infty}\rho(M_k')$.
\end{conj}

\begin{conj}
Theorem~\ref{thm-bdim} holds if we replace the condition (A6) by (A6').
\end{conj}

\begin{ques}
Can we obtain the box dimension of generalized affine FIFs under the conditions (A1)-(A3) and (A5)?
\end{ques}

\begin{center}
{\noindent\bf Acknowledgments}
\end{center}
\medskip

The authors wish to thank Dr. Jian-Ci Xiao for helpful discussions.

\bibliographystyle{amsplain}

\begin{thebibliography}{11}


\bibitem{BHR19}{\sc B.~B\'{a}r\'{a}ny, M.~Hochman and A.~Rapaport}, {\em Hausdorff dimension of planar self-affine sets and measures}. Invent. Math., {\bf 216} (2019), 609--659.

\bibitem{BRS20}{\sc B.~B\'{a}r\'{a}ny, M.~Rams and K.~Simon}, {\em Dimension of the repeller for a piecewise expanding affine map}. Ann. Acad. Sci. Fenn. Math., {\bf 45} (2020), 1135--1169.


\bibitem{Bar86} {\sc M.~F. Barnsley}, {\em Fractal functions and interpolation},  Constr. Approx., {\bf 2} (1986), 303--329.


\bibitem{BEH89}
{\sc M.~F. Barnsley, J.~H. Elton, and D.~P. Hardin}, {\em Recurrent iterated function systems}, Constr. Approx., {\bf 5} (1989), 3--31.

\bibitem{BEHM89}
{\sc M.~F. Barnsley, J. Elton, D. Hardin and P.~Massopust}, {\em Hidden variable fractal interpolation functions}, SIAM J. Math. Anal., {\bf 20} (1989), 1218--1242.




\bibitem{BaMa15} {\sc M.~F. Barnsley and P.~R. Massopust}, {\em Bilinear fractal interpolation and box dimension}, J. Approx. Theory, {\bf 192} (2015), 362--378.


\bibitem{BDD06} {\sc P. Bouboulis, L. Dalla and V. Drakopoulos}, {\em Construction of recurrent bivariate fractal interpolation surfaces and computation of their box-counting dimension}, J. Approx. Theory, {\bf 141} (2006), 99--117.

\bibitem{CKV15} {\sc A.~K.~B. Chand, S.~K. Katiyar, and P. V. Viswanathan}, {\em Approximation using hidden variable fractal interpolation function}, J. Fractal Geom., {\bf 2} (2015), 81--114.




\bibitem{Fal90} {\sc K.~J. Falconer}, {\em Fractal geometry: Mathematical foundation and applications}, New York, Wiley, 1990.

\bibitem{FalKem16} {\sc K.~Falconer and T.~Kempton}, {\em Planar self-affine sets with equal Hausdorff, box and affinity dimensions}, Ergod. Theory Dyn. Syst., {\bf 38} (2018), 1369--1388.

\bibitem{Fen08} {\sc Z.~Feng}, {\em Variation and Minkowski dimesnsion of fractal interpolation surfaces}, J. Math. Anal. Appl., {\bf 345} (2008), 322--334.

\bibitem{HorJoh90} {R. ~A. Horn and C. ~R. Johnson}, {\em Matrix analysis (Second Edition)}, Cambridge University Press, 2013.

\bibitem{JhaVer} {\sc S.~Jha and S.~Verma}, {\em Dimensional analysis of $\alpha$-fractal functions}, Results Math., {\bf 76} (2021), Paper No. 186,  24 pp.

\bibitem{KRZ18}
{\sc Q.-G.~Kong, H.-J.~Ruan and S.~Zhang},
{\em Box dimension of bilinear fractal interpolation surfaces},
Bull. Aust. Math. Soc., {\bf 98} (2018), 113--121.

\bibitem{LR21} {\sc Z. Liang and H.-J. Ruan}, {\em Construction and box dimension of recurrent fractal interpolation surfaces}, J. Fractal Geom., {\bf 8} (2021), 261--288.


\bibitem{MaHa92} {\sc D.~S. Mazel and M.~H. Hayes},  {\em Using iterated function systems to model discrete sequences}, IEEE Trans. Signal Process, {\bf 40} (1992), 1724--1734.

\bibitem{Nav10} {\sc M. A.~Navascu\'{e}s}, {\em Fractal approximation}, Complex Anal. Oper. Theory, {\bf 4} (2010), 953--974.



\bibitem{RenShen21} {\sc H.~Ren and W.~Shen}, {\em A dichotomy for the Weierstrass-type functions}, Invent. Math., {\bf 226} (2021), 1057--1100.

\bibitem{RSY09} {\sc H.-J. Ruan, W.-Y. Su and K. Yao},  {\em Box dimension and fractional integral of linear fractal interpolation functions}, J. Approx. Theory, {\bf 161} (2009), 187--197.

\bibitem{RXY21} {\sc H.-J. Ruan, J.-C. Xiao and B. Yang}, {\em Existence and box dimension of general recurrent fractal interpolation functions}, Bull. Aust. Math. Soc., 103 (2021), 278--290.



\bibitem{VDL94} {\sc J.~L. V\'{e}hel, K. Daoudi and E. Lutton}, {\em Fractal modeling of speech signals}, Fractals, {\bf 2} (1994), 379--382.

\bibitem{WY13} {\sc H.~Y. Wang and J.~S. Yu}, {\em Fractal interpolation functions with variable parameters and their analytical properties}, J. Approx. Thoery, {\bf 175} (2013), 1--18.


\end{thebibliography}

\end{document}